\documentclass[12pt,a4paper]{amsart}

\usepackage[usenames]{color}
\usepackage{times}
\usepackage{amsfonts}
\usepackage{amsthm}
\usepackage{amsmath}
\usepackage{psfrag}

\usepackage{times}
\usepackage{latexsym,color}
\usepackage{amssymb}
\usepackage{graphicx}

\usepackage{epsfig}

\usepackage[mathscr]{euscript}
\usepackage{tikz}
\usetikzlibrary{positioning}
\usepackage{fullpage,parskip}
\usepackage[all]{xy}
 \usepackage[hang]{caption}
 \usetikzlibrary{snakes}
\usetikzlibrary{arrows}

\advance \topmargin by -\headheight
\advance \topmargin by -\headsep
\evensidemargin \oddsidemargin
\newcommand{\ints}{{\mathbb Z}}
\newcommand{\reals}{{\mathbb R}}

\newcommand{\A}{{\mathbb A}}

\newtheorem{theorem}{Theorem}
\newtheorem{lemma}{Lemma}
\newtheorem{corollary}{Corollary}
\newtheorem{proposition}{Proposition}

\newtheorem{example}{Example}
\newtheorem{definition}{Definition}
\newtheorem{remark}{Remark}

\begin{document}

\title[Order symmetry \& orthogonality of trajectories in discrete iets]{Order symmetry and orthogonality of trajectories in discrete interval exchange transformations}

\author[S.S. Ferenczi]{S\'ebastien S. Ferenczi}
\address{Aix Marseille Universit\'e, CNRS, Centrale Marseille, Institut de Math\' ematiques de Marseille, I2M - UMR 7373\\13453 Marseille, France.}
\email{ssferenczi@gmail.com}

\author[L.Q. Zamboni]{Luca Q. Zamboni}
\address{Institut Camille Jordan\\
Universit\'e Claude Bernard Lyon 1\\
43 boulevard du 11 novembre 1918\\
F69622 Villeurbanne Cedex
(France)}
\email{zamboni@math.univ-lyon1.fr}

\keywords{Discrete interval exchange transformations, clustering, Burrows-Wheeler transform}

\date{}

\subjclass[2010]{Primary 68R15, Secondary 37B10}

\begin{abstract} {Let $\pi=(<_D,<_A)$ be a pair of distinct orders on a $k$-letter alphabet $\A. $  The periodic trajectories $v_i^\infty $ in the natural coding of a discrete $k$-interval exchange transformations $T$ with permutation $\pi$ are characterized by the following order symmetry : $v_i^\omega <_D v_j^\omega$ (lexicographically) if and only if $v_i^{-\omega}<_Av_j^{-\omega}$ (reverse lexicographically) where we write $u^\omega=uuu\cdots$ and $u^{-\omega}=\cdots uuu.$   
For general words $u$ and $v$ over $\A$, the orders $<_D$ and $<_A$ need not agree in which case one has either $u^\omega<_A v^\omega<_D u^\omega$ (Type 1) or $v^\omega<_A u^\omega<_D v^\omega$ (Type 2).  We partition all such order crossings  amongst the set of conjugates of two words $u$ and $v$  according to their type into disjoint families  $T_1(u,v)$ and $T_2(u,v)$ and define the index $i(u,v)$ as the sum of the cardinalities $|T_1(u,v)|+|T_2(u,v)|.$ Remarkably the difference, $|T_2(u,v)|-|T_1(u,v)|,$ depends only on the abelianisation of $u$ and $v,$ i.e., on the Parikh vectors $\lambda(u)$ and $\lambda(v).$ In fact, we show that  $|T_2(u,v)|-|T_1(u,v)|=\lambda(u)^T \Omega \lambda(v)$ where $\Omega$ is a $k\times k$ skew symmetric  matrix depending only on the permutation $\pi.$  If $\pi$ is symmetric and $k$ is even, then the associated bilinear form $\Delta : \reals^\A \times \reals^\A \rightarrow \reals$ defined by $\Omega$ is non-degenerate and hence  symplectic.  It follows that the Parikh vectors of the trajectories of a discrete $k$-interval exchange transformation with permutation $\pi$  are  pairwise orthogonal with respect to $\Omega.$  
Applied to dimension $3,$ we obtain a simple arithmetic formula on the length vector for the number of orbits in a discrete $3$-interval exchange transformation $T$ and hence a characterization of minimality.  If $T$ is symmetric, our condition coincides with a result originally due to Pak and Redlich.  For general $k,$ the orthogonality of the trajectories imposes conditions on the number of distinct trajectories. We show that a discrete  $k$-interval exchange transformation $T$ has at most $\lfloor \frac{k+d}2 \rfloor$ distinct trajectories where $d=\dim \ker (\Omega).$  If $T$ is symmetric, then  the number of distinct trajectories is at most$\lfloor \frac{k+1}2 \rfloor.$ An alternate interpretation of this result is that on an ordered $k$-letter alphabet, there are at most $\lfloor \frac{k+1}2 \rfloor$ primitive, pairwise non conjugate perfectly clustering words (not necessarily of the same length) which perfectly cluster collectively in a single array in which all their conjugates are arranged in increasing order. The rank of $\Omega$ depends only on the Rauzy class of $\pi$ and hence the above bounds apply to the entire Rauzy class of $\pi.$  }\end{abstract} 
\maketitle 

\section{Introduction}

Languages generated by the  trajectories of interval exchange transformations possess many rich combinatorial features (see for example \cite{DoHu,fz4,fz5,fz6,fhuz,lap,lapplou,lapreut,mrs} in addition to the vast existing literature devoted to the combinatorics of Sturmian languages). 
%We represent the periodic trajectories of $T$ by  bi-infinite periodic words $u^\infty=\cdots uuu\cdots$ where $u$ is a primitive word over the alphabet $\A.$ We write $u^\omega=uuu\cdots$  (resp.,  $u^{-\omega}=\cdots uuu)$ for the right (resp., left) infinite word determined by $u.$  
In this paper we investigate the following order symmetry: Let $\pi=(<_D,<_A)$ be a pair of distinct orders on a $k$-letter alphabet $\A,$  and $T$ a discrete $k$-interval exchange transformation with permutation $\pi.$ Then given  (periodic) trajectories $v_i^\infty,v_j^\infty$ of $T$ we have:   $v_i^\omega <_D v_j^\omega$ (lexicographically) if and only if $v_i^{-\omega}<_Av_j^{-\omega}$ (reverse lexicographically) where we write $u^\omega=uuu\cdots$ and $u^{-\omega}=\cdots uuu.$   
For simplicity we shall write $u<_D v$ for $u^\omega <_D v^\omega$ and $u<_A v$ for $u^{-\omega}<_A v^{-\omega}.$ 
When $v_i$ and $v_j$ are cyclically conjugate to one another, this order symmetry may be described in terms of the {\it clustering} of the Burrows-Wheeler transform of $u=v_i$ (see \cite{bw}). More precisely,   if the cyclic conjugates of $u$ are arranged in increasing lexicographic order with respect to $<_D,$ then the letters in the last column of this array, defined as the Burrows-Wheeler transform of $u,$ are arranged in increasing order with respect to $<_A$  (see for example \cite{fz4, lap, mrs} and also \cite{GR}).

Accordingly we say that a family $\mathcal F=(v_1,v_2,\ldots ,v_n)$  of words over $\A$ is {\it $\pi$-clustering} if and only if for all conjugates $v'_i$ of $v_i$ and $v'_j$ of $v_j$ with $1\leq i,j\leq n,$ we have $v'_i<_D v'_j\Leftrightarrow v'_i<_A v'_j.$ 
We say $\mathcal F$ is {\it perfectly clustering} if $\mathcal F$ is $\pi$-clustering for some pair of {\it symmetric orders }  meaning $a<_D b \Leftrightarrow b<_A a$ for all  $a,b\in \A.$
Thus the  trajectories of a discrete $k$-interval exchange transformation with permutation $\pi$ form a $\pi$-clustering family. And conversely,  every $\pi$-clustering family is made up of the trajectories of a discrete $k$-interval exchange transformation with permutation $\pi.$ Associated to each $\pi$-clustering family  $\mathcal F=(v_1,v_2,\ldots ,v_n)$  is a common lexicographic array in which the conjugates of the $v_i$ are simultaneously arranged in increasing order with respect to both $<_D$ and $<_A.$ This common lexicographic array may be viewed as a merging of the $n$ Burrows-Wheeler arrays of the individual $v_i,$ a merging which in this case preserves clustering. 

In general however, the merging of the Burrows-Wheeler arrays of  two $\pi$-clustering words $u$ and $v$ need not preserve clustering.  In which case there exist conjugates $u'$ of $u$ and $v'$ of $v$ such that either $u'<_A  v' <_D u'$ (Type 1) or $v'<_A u' <_D v'$ (Type 2). If we represent each of $u'$ and $v'$ by a horizontal line segment and order the left end points according to $<_D$ and the right end points according to $<_A,$ then  the line segments would cross one another. Accordingly we say the pair $(u',v')$ defines a Type 1 or Type 2 {\it order crossing}. We introduce an equivalence relation on the set of all Type 1 (resp. Type 2) order crossings, and let $T_i(u,v)$ be the set of equivalence classes of all Type i order crossings for $i=1,2$  (see equations (\ref{type1}) and (\ref{type2})).

We define the {\it index} \[i(u,v)= |T_1(u,v)|+|T_2(u,v)|\] as a measure of the extent to which the family $\mathcal F=(u,v)$ fails to $\pi$-cluster when $u$ and $v$ are each $\pi$-clustering words. If $\mathcal F=(u,v)$ is not $\pi$-clustering, then either $T_1(u,v)$ or $T_2(u,v)$ (or both) is non empty, and thus $i(u,v)>0.$ While if  $\mathcal F=(u,v)$ is $\pi$-clustering, then there are no order crossings of either type and hence $T_1(u,v)$ and $T_2(u,v)$ are both empty and  $i(u,v)=0.$ We say $u$ and $v$ have {\it mixed order type} if $T_1(u,v)$ and $T_2(u,v)$ are each non empty. We show that this is not possible on a binary alphabet, that is when $u$ and $v$ are finite Sturmian words, while there exist examples on a ternary alphabet.

The index $i(u,v)$ depends strongly on the choice of words $u$ and $v$ and can be laborious to compute in practice.  Remarkably the difference $|T_2(u,v)|-|T_1(u,v)|$ depends only on the abelianisation  of $u$ and $v,$ i.e.,  on the Parikh vectors $\lambda (u)$ and  $\lambda(v),$  and moreover is easily computable. 

\begin{theorem}\label{tA}
For any pair of words $(u,v)$ over $\A$ 
\[|T_2(u,v)|-|T_1(u,v)|=\lambda (u)^T \Omega \lambda(v)\]
where  $\Omega=(\Omega_{i,j})_{i,j\in \A}$  is the $k\times k$ skew symmetric matrix 
defined by

\begin{align}\label{Omegaij} \Omega_{i,j}= \begin{cases} +1\,\,\,\, \text{if  $j<_A i$ and $i<_D j$}\\ 
-1\,\,\,\, \text{if  $j<_D i$ and $i<_A j$}\\\,\,\,0\,\,\,\,\, \text{otherwise.} \end{cases}
\end{align}
\end{theorem}

Theorem~\ref{tA} provides a simple formula for computing the index for any pair of words $u$ and $v$ not having mixed order type. In this case either $T_1(u,v)$ or $T_2(u,v)$ (or both)  is empty, and hence \[i(u,v)= |\lambda (u)^T \Omega \lambda(v)|.\]
This applies in particular to any pair of Sturmian words $u$ and $v.$ 
On a binary alphabet   
\[\Omega=\begin{pmatrix}\,\,\,\,0&1\\-1&0\end{pmatrix},\]
and thus writing $\lambda(u)=(a,b)$ and $\lambda(v)=(c,d),$ we obtain 

\[i(u,v)=|\begin{pmatrix} a & b \end{pmatrix} \begin{pmatrix}\,\,\,\,0&1\\-1&0\end{pmatrix} \begin{pmatrix} c\\d \end{pmatrix}|=|ad-bc|\] 
for the index of two Sturmian words $u$ and $v.$

As another consequence of Theorem~\ref{tA}, if $i(u,v)=0,$ then  the Parikh vectors of $u$ and $v$ are orthogonal with respect to $\Omega.$  In particular, if 
 $\mathcal F=(v_1,v_2,\ldots ,v_n)$ is a $\pi$-clustering family, then $\lambda(v_i)^T\Omega \lambda (v_j)=0$ for all $i,j.$ These orthogonality conditions in turn impose a constraint on the size $n$ of the family:

\begin{corollary}\label{cA}
Let $\mathcal F=(v_1,v_2,\ldots ,v_n)$ be a family consisting of $n$ pairwise non-conjugate primitive words $v_i$ over $\A.$ If $\mathcal F$ is $\pi$-clustering, then $n\leq \lfloor \frac {k+d}{2} \rfloor $ 
where $d=\dim \ker (\Omega).$ In particular if $\Omega$ is invertible, in which case $k$ is even,   then $n\leq k/2.$ For $\pi$ symmetric,  $\Omega$ is invertible if and only if $k$ is even. If $\mathcal F=(v_1,v_2,\ldots ,v_n)$ is  perfectly clustering,  then $n\leq \lfloor \frac {k+1}{2} \rfloor. $ 
\end{corollary}

%Moreover  if $\Omega$ is invertible, in which case $k$ is even,   then $n\leq k/2.$ For a symmetric  $\pi,$ the matrix $\Omega$ is the $k\times k$ matrix with $0$'s along the diagonal, $+1$'s above the diagonal and $-1$'s below the diagonal. It follows easily by induction on $k$ that $\det \Omega =1$ for each $k$ even. Thus in the symmetric case, $\Omega$ is invertible if and only if $k$ is even. We use this to show that  if $\mathcal F=(v_1,v_2,\ldots ,v_n)$ is a perfectly clustering family consisting of $n$-pairwise non-conjugate primitive words $v_i$ over $\A,$ then $n\leq \lfloor \frac {k+1}{2} \rfloor. $ 

These results have related applications to the trajectories of discrete $k$-interval exchange transformations. 
In fact, %a primitive word $u$ over $\A$ is $\pi$-clustering  if and only if $u$ is the natural coding of a discrete minimal (i.e.,  consisting of a single orbit) $k$-interval exchange transformation $T$ defined by the  $\pi=(<_D, <_A)$ (see \cite{fz4}). we show that a family $\mathcal F=(v_1,v_2,\ldots ,v_n)$ of primitive words over $\A$ is $\pi$-clustering if and only if there exists a discrete $k$-interval exchange transformation $T$ for the pair $\pi=(<_D, <_A)$ whose natural coding decomposes into $n$-trajectories  $v_1^\infty, v_2^\infty, \ldots ,v_n^\infty.$ %The length vector of $T$ is the sum of the Parikh vectors of the $v_i.$ 
by Theorem~\ref{tA}, the Parikh vectors of the trajectories of a discrete $k$-interval exchange transformation $T$ are orthogonal with respect to $\Omega.$ Applied to dimension $3,$ we show:

\begin{corollary}\label{cB} Let $T$ be a discrete $3$-interval exchange transformation with pair of orders $\pi.$  Then the number of orbits of $T$ is equal to $\gcd \Omega \lambda$ where $\lambda$ is the length vector of $T.$  Hence $T$ is minimal if and only if  
$\gcd \Omega \lambda=1.$
\end{corollary}

 If $T$ is symmetric with length vector $\lambda=(a,b,c),$ we obtain $T$ is minimal if and only if 
\[\gcd \begin{pmatrix}0&1&1\\-1&0&1\\-1&-1&0 \end{pmatrix}\begin{pmatrix}a\\b\\c \end{pmatrix}=1 \Leftrightarrow \gcd (b+c,c-a, -a-b)=1\]
or equivalently $\gcd (b+c,a+b)=1.$ This coincides with a result due to I. Pak and A. Redlich in \cite{PR}.

As in the case of $\pi$-clustering families, Theorem~\ref{tA} yields upper bounds on the number of distinct trajectories in a discrete $k$-interval exchange transformation with pair of orders $\pi.$ 

\begin{corollary}\label{cC} Let $T$ be a discrete $k$-interval exchange transformation with pair of orders $\pi.$  Then $T$ admits at most $\lfloor \frac {k+d}{2} \rfloor$ distinct trajectories, where $d=\dim \ker (\Omega).$ In particular, a discrete symmetric $k$-interval exchange transformation admits at most $\lfloor \frac {k+1}{2} \rfloor$ distinct trajectories. \end{corollary}

It turns out that the rank of $\Omega$ is an invariant of the Rauzy class of $\pi$ and hence the above upper bounds, on the size of primitive $\pi$-clustering families and  on the number of distinct trajectories in a  discrete $k$-interval exchange transformation, apply throughout the entire Rauzy class of $\pi.$

The matrix $\Omega$ is defined in the same way  for a continuous $k$-interval exchange transformation $T$ with   pair of orders $\pi=(<_D,<_A).$ The bilinear form defined by $\Omega$
plays a part in the geometric theory of interval exchanges developed in \cite{vee}, to which we refer the reader to the courses  \cite{via} and \cite{yoc}. It is degenerate in general but often called symplectic by abuse of language, as it defines a symplectic form, called the {\it intersection form}, on the first cohomology group $H^1(M,\reals)$ of the  translation surface $M$  built from $T$ by the method of {\it zippered rectangles}, whereby the homology group $H_1(M,\reals)$  is identified with  $\reals^\A/\ker\Omega.$

\section{Clustering, order condition and interval exchange transformations}

Given a finite nonempty word $u$ over an ordered  $k$-letter alphabet $\A,$  we may arrange the cyclic conjugates of $u$  in increasing lexicographic order thereby forming a rectangular array of words with the lexicographically  smallest at the  top and the  largest at the bottom. The Burrows-Wheeler transform of $u$ is defined as the last column of this array of words \cite{bw}. This process defines a  lossless compression algorithm in that the original word $u$ may be completely recovered up to cyclic conjugacy from its Burrows-Wheeler transform.    A word $u$ is said to {\it cluster} if the Burrows-Wheeler transform of $u$  is of the form $b_1^{m_1}b_2^{m_2}\cdots b_j^{m_j}$ with $m_i\geq 1$ and  pairwise distinct letters $b_i.$     In other words, a word $u$ clusters if and only if the last column of the lexicographic array of $u$ is obtained by some permutation of the letter blocks in the first column. Thus each clustering word $u$ over an ordered alphabet $\A$ defines a second order  on  $\A$ determined by the order of the letter blocks in the last column. 

Alternatively, as shall be our convention in the present paper, we  fix a priori a  pair of distinct orders $\pi= (<_D, <_A)$ on the alphabet $\A $, called the {\it departure} and  {\it arrival} orders respectively, and  say that a word $u$ is $\pi$-clustering if by arranging the conjugates of $u$ in lexicographic order with respect to the departure order $<_D,$ the letters in the last column are arranged according to the arrival order $<_A.$ 
In this case more is true: In this array, the conjugates of $u$ are arranged in  lexicographic order with respect to $<_D$ and in reverse lexicographic order with respect to $<_A.$ In fact, if $u',u''$ are two conjugates of $u$ with $u'<_D u'',$ then writing $u'=xaz$ and $u''=ybz$ with $a,b$  distinct letters, since $zxa<_D zyb,$ it follows that $a<_A b$ and hence $u'<_A u''$ as required.

We use this second more symmetric characterization of clustering to extend the notion to families of non empty words over $\A.$ For this purpose, we extend both $<_D$ and $<_A$ to partial orders over the free semigroup $\A^+$ as follows: For $u,v \in \A^+,$ we write $u<_D v$ if and only if $u^\omega<_D v^\omega,$ i.e., $u^\omega=uuu\cdots$  is lexicographically smaller than $v^\omega=vvv\cdots$ with respect to $<_D.$ Analogously we write $u<_A v$ if and only if $u^{-\omega}<_A v^{-\omega},$ i.e., $u^{-\omega}=\cdots uuu$  is reverse lexicographically smaller than $v^{-\omega}=\cdots vvv$ with respect to $<_A.$ Thus when comparing two words $u$ and $v$ with respect to $<_D$ we read from left to right, while comparing them with respect to $<_A$ we read from right to left.

\begin{definition}Let $\mathcal F=(v_1,v_2,\ldots ,v_n)$ be a family of  (non empty) words over $\A.$ We say $\mathcal F$ is $\pi$-clustering if for all conjugates $u$ of $v_i$ and $v$ of $v_j$ with $1\leq i,j\leq n,$ we have $u<_A v\Leftrightarrow u<_D v.$
 \end{definition}
 
 For example, the family $\mathcal F=(ac,b)$ is perfectly clustering on the ternary alphabet $\{a,b,c\}.$  It is maximal in the sense that the only way to extend it to a larger perfectly clustering family is to adjoin powers of conjugates of $ac$ or $b.$ We shall later see that it is also maximal in that there does not exist a perfectly clustering family consisting of more than two primitive pairwise non-conjugate words on a ternary alphabet.

A key feature linking clustering words  and natural codings of interval exchange transformations is the following  {\it order condition} on a language $L\subset \A^*$ :

\begin{definition} A  language $L$ over $\A$ satisfies the $\pi$-{\it order condition} if whenever $axb, cxd\in L$ with $a,b,c,d \in \A,$  $x\in \A^*,$  $a\neq c$ and $b\neq d,$ we have $a<_A c \Leftrightarrow b<_D d.$ 
\end{definition}

%For example, if we take $u=ac$ and $v=b$ relative to the symmetric orders $a<_Db<_D c$ and $c<_A b<_A a,$ then one easily verifies that $L_u\cup L_v$ satisfies the $\pi$-order condition.  In contrast, let  $u=a$ and $v=b$ with $a\neq b$ relative to the symmetric orders $a<_D b$ and $b<_A a,$  then the languages $L_u$ and $L_v$ satisfy the $\pi$-order condition.  However, $L_u\cup L_v$ does not satisfy the $\pi$-order condition since $aa\in L_u$ and $bb\in L_v.$   

Given a word $u,$ let $L_u$ denote the language consisting of all factors of $u^\omega.$ Given a family   $\mathcal F=(v_1,v_2,\ldots ,v_n)$ of words over $\A$ we set $L_{\mathcal F}=\bigcup _{i=1}^nL_{v_i}.$

\begin{lemma}\label{combine}Let $\mathcal F=(v_1,v_2,\ldots ,v_n)$ be a family of  words over $\A.$    Then $\mathcal F$ is $\pi$-clustering if and only if the  language $L_{\mathcal F}$ satisfies the $\pi$-order condition. 
\end{lemma}

\begin{proof}Assume $L_{\mathcal F}$ satisfies the $\pi$-order condition and let $u,v$ be conjugates of $v_i$ and $v_j$ respectively with $u<_A v.$ Comparing $u^{-\omega}$ with $v^{-\omega}$ we can write $u^{-\omega} =\cdots ax$ and $v^{-\omega}= \cdots cx$ for some $x\in \A^*$ and $a<_A c.$
Similarly we can write $u^\omega=yb\cdots $ and $v^\omega=yd \cdots $ for some $y\in \A^*$ and $b\neq d.$ Thus $axyb\in L_u=L_{v_i}$ and $cxyd\in L_v=L_{v_j}.$ Since  $L_u\cup L_v\subseteq L_{\mathcal F}$ and 
$L_{\mathcal F}$ satisfies the $\pi$-order condition it follows that $b<_D d$ and hence $u<_D v$ as required. 
Conversely, if $L_{\mathcal F}$ does not satisfy the order condition, then there exist $axb,cxd\in L_{\mathcal F}$ such that $a<_A c$ and $d<_D b.$  Pick $1\leq i,j\leq n$ such that   $axb\in L_{v_i}$ and $cxd\in L_{v_j}.$ Let $u$ be the conjugate of $v_i$ beginning after the $a$ in an occurrence of $axb$ in $v_i^\omega$ and let $v$ be the conjugate of $v_j$ beginning after the $c$ in an occurrence of $cxd$ in $v_j^\omega.$   Then $u<_A v$ since $a<_A c$ while  $v<_D u$ since $v^\omega$ begins in $xd$ and $u^\omega$ in $xb$ and $d<_D b.$  \end{proof}

A {\it discrete $k$-interval exchange transformation} $T$  with pair of orders (or permutation) $\pi= (<_D, <_A)$  is given by a mapping $T:I\rightarrow I$ on an integer interval $I=\{1,2,\ldots,m\}$  partitioned into $k$ subintervals labeled by the letters in $\A$ in increasing order with respect to $<_D$   while their images  under $T,$ which also partition $I,$ are labeled in increasing order with respect to  $<_A,$ and the restriction of $T$ to each labeled subinterval is a translation. We represent the natural coding of each trajectory by a bi-infinite word $v^\infty=\cdots vvv\cdots$ with $v$ a primitive word over $\A.$

 \begin{proposition}\label{fequiv2} Let $\mathcal F=(v_1,v_2,\ldots ,v_n)$ be a family of primitive words over $\A,$ and let $T$ denote the discrete $k$-interval exchange transformation with pair of orders $\pi$ and length vector $\lambda (v_1)+\cdots + \lambda (v_n).$ If $\mathcal F $ is $\pi$-clustering, then the  natural coding of $T$ decomposes into $n$ trajectories $v_1^{ \infty}, v_2^{ \infty} , \ldots , v_n^{ \infty}.$ Conversely, if the natural coding of a discrete $k$-interval exchange transformation $T$ with pair of orders $\pi$ decomposes into $n$ trajectories $v_1^{ \infty}, v_2^{ \infty} , \ldots , v_n^{ \infty},$   then  the family $\mathcal F=(v_1,v_2,\ldots ,v_n)$ is $\pi$-clustering. 
\end{proposition}

\begin{proof}  Assume first $\mathcal F=(v_1,v_2,\ldots ,v_n)$ is $\pi$-clustering. We consider the lexicographic array of $\mathcal F$ consisting of the conjugates of the $v_i$ arranged in increasing order with respect to $<_D$ and $<_A.$    Although the row lengths of this  array may vary depending on the lengths of the $v_i,$  we shall regard the first letters in each row as making up the first column of the array, and the last letters in each row as making up the last column.  
Thus, the letter blocks in the first column are arranged in increasing order with respect to $<_D$ while in the last column they are ordered according to $<_A.$  Let $T$ denote the discrete interval exchange transformation with pair of orders $\pi$ and length vector $\lambda (v_1)+\lambda (v_2)+\cdots +\lambda (v_n).$ We may regard $T$ as  a mapping  which sends the first column of the lexicographic array of $\mathcal F$ onto the last column by translating each letter block in the first column onto the corresponding letter block in the last column. For each letter $a\in \A,$ the $a$-block in the first column corresponds to the cylinder $[a]$ consisting of all conjugates of  the $v_i$ beginning in $a$ arranged in increasing order with respect to   $<_D$ while the $a$-block in the last column $T[a]$ consists of all conjugates of the $v_i$ terminating in $a$ arranged in increasing order with respect to  $<_A.$ 
Thus by translating the $a$-block in the first column to the $a$-block in the last column, $T$ maps each row $x\in [a]$  to a row $T(x)\in T[a]$ and for all $x,y \in [a],$ if $x<_Dy$ then $T(x)<_A T(y).$
We claim that $T(x)=a^{-1}xa$ for  each $x\in [a],$ i.e., $T(x)$ is the conjugate of $x$ obtained by cycling the initial letter $a$ to the end of $x.$  In fact for all $x,y \in [a],$ if $x<_Dy$ then $x<_A y.$ Since both $x$ and $y$ begin in the same letter $a,$ it follows that $a^{-1}xa<_A a^{-1}ya.$ Thus if $[a]=\{x_1, x_2, \ldots ,x_m\}$ with $x_1<_D x_2<_D \cdots <_D x_m,$ then $T[a]=\{a^{-1}x_1a, a^{-1}x_2a,  \ldots , a^{-1}x_ma\}$ with $a^{-1}x_1a<_Aa^{-1}x_2a<_A \cdots <_A a^{-1}x_ma.$  It follows that  $T(x)=a^{-1}xa$ for each $x\in [a].$ In other words, the process of translating each letter block in the first column to the corresponding letter block in the last column preserves conjugacy classes and amounts to the usual shift map consisting in  cycling the initial letter of each word to the end of the word. It follows that the natural coding of $T$ decomposes into $n$ trajectories $v_1^{ \infty}, v_2^{ \infty} , \ldots , v_n^{ \infty}.$

For the converse,   assume the natural coding of a $k$-interval exchange transformation $T$ with pair of orders $\pi$ decomposes into $n$ trajectories  $v_1^\infty, v_2^\infty,\ldots ,v_n^\infty.$  Set $\mathcal F=(v_1,v_2,\ldots ,v_n).$ Then by Theorem 19 in \cite{fhuz}, the language $L_{\mathcal F}$ generated by the trajectories satisfies the $\pi$-order condition and hence by Lemma~\ref{combine}, the family $\mathcal F$ is $\pi$-clustering.   \end{proof}

\begin{proposition}\label{li} Let $\mathcal F=(v_1,v_2,\ldots ,v_n)$ be a family of  primitive words over $\A.$ If $\mathcal F$ is $\pi$-clustering, then  the set of Parikh vectors  $P_{\mathcal F}= \{\lambda (v_i)\,|\,1\leq i\leq n\}$ is linearly independent in $\reals ^\A.$ \end{proposition}

\begin{proof} Assume to the contrary that there exists a non trivial real dependence relation between the vectors in $P_{\mathcal F}.$ Then there exists a non trivial rational dependence relation and after clearing denominators, we obtain a non trivial integral dependence relation between the vectors in $P_{\mathcal F}.$ By grouping  terms according to signs we can write 
\begin{equation}\label{dep}r_1\lambda (v_{i_1}) +\cdots + r_p\lambda (v_{i_p})=s_1\lambda (v_{j_1}) +\cdots + s_q\lambda (v_{j_q})\end{equation} for some  $r_i,s_j$ strictly positive integers and $\lambda (v_{i_1}), \ldots ,\lambda (v_{i_p}), \lambda (v_{j_1}) ,\ldots , \lambda (v_{j_q})$ pairwise distinct vectors in $P_{\mathcal F}.$  Let $T$ denote the discrete interval exchange transformation with pair of orders $\pi$ and  length vector $r_1\lambda (v_{i_1}) +\cdots + r_p\lambda (v_{i_p}).$ Applying Proposition~\ref{fequiv2}  to the family $(v_{i_1},\ldots , v_{i_1}, \ldots , v_{i_p}, \ldots , v_{i_p})$, whose language being a sub language of $L_{\mathcal F}$ satisfies the $\pi$-order condition,  it follows  that the trajectories of $T$ are $v_{i_1}^\infty , \ldots , v_{i_p}^\infty$ with multiplicities $r_1,\ldots ,r_p$ respectively. On the other hand applying Proposition~\ref{fequiv2} to the  family $(v_{j_1},\ldots , v_{j_1}, \ldots , v_{j_q}, \ldots , v_{j_q}),$ 
the trajectories of $T$ are $v_{j_1}^\infty, \ldots , v_{j_q}^\infty$ with multiplicities $s_1,\ldots ,s_q.$ This in turn is a contradiction as each  $v_i$ is primitive. \end{proof}

%\begin{remark}\rm If $T$ is a discrete $k$-interval exchange transformation with trajectories $v_1^\infty, v_2^\infty, \ldots ,v_n^\infty,$ with $v_i$ primitive words. By Proposition~\ref{fequiv} the family $\mathcal F =(v_1,v_2,\ldots ,v_n)$ is $\pi$-clustering and hence by Proposition~\ref{li} the set of Parikh vectors $\{l(v_i)\,|\,1\leq i\leq n\}$ is a linearly independent subset of $\reals ^k$ and hence of cardinality at most $k.$ \end{remark} 

We say a family $\mathcal F=(v_1,v_2,\ldots ,v_n)$ is {\it primitive} if each $v_i$ is primitive and the $v_i$ are pairwise non conjugate words.

\begin{corollary}\label{span} Let $\mathcal F=(v_1,v_2,\ldots ,v_n)$ be a primitive family of words over $\A.$ If $\mathcal F$ is $\pi$-clustering, then the set of Parikh vectors  $P_{\mathcal F}= \{\lambda (v_i)\,|\,1\leq i\leq n\}$ spans an $n$-dimensional subspace of $\reals^\A.$  In particular $n\leq k.$ 
\end{corollary} 

\begin{proof} Assume $\mathcal F$ is $\pi$-clustering and hence in particular each $v_i$ is $\pi$-clustering. It follows that  the Parikh vectors $\lambda (v_i)$ are pairwise distinct. In fact,  if $\lambda (v_i)=\lambda (v_j),$ then as each is $\pi$-clustering it follows that $v_i$ and $v_j$ are conjugate to one another. As $\mathcal F$ is primitive, this implies $i=j.$   Thus by Proposition~\ref{li}, the subspace of $\reals^\A$ spanned by the vectors $\lambda (v_1),\lambda (v_2),\ldots,\lambda (v_n)$ is of dimension $n$ and hence $n\leq k.$\end{proof}

\section{Index and orthogonality}     

Let $u$ and $v$ be two $\pi$-clustering words over $\A$ and set  $\mathcal F=(u,v).$ If $\mathcal F$ is $\pi$-clustering, then the cyclic conjugates of $u$ and $v$ may be simultaneously ordered in a single lexicographic array with respect to both the departure and arrival orders. If we represent each conjugate of $u$ and $v$ by a horizontal line segment, and we order the left endpoints of each segment according to $<_D$ and the right endpoints according to $<_A,$ then the line segments are non-crossing. In contrast, if $\mathcal F$ is not $\pi$-clustering, then there will be at least one  crossing between a line segment representing a conjugate of $u$ and  a line segment representing a conjugate of $v.$  This is illustrated in Figure 1 for the words $u$ and $v$ given in Example~\ref{Simon}. Such crossings are naturally partitioned into two disjoint {\it types}, corresponding to the two ways of labeling the two diagonals in an $X$ by $u$ and $v.$ We shall be interested in counting the number of such crossings of each type.

More precisely, assume $\mathcal F=(u,v)$ is not $\pi$-clustering.  
Then there exist conjugates $u'$ of $u$ and $v'$ of $v$ such that either $u'<_A  v' <_D u'$ (Type 1) or $v'<_A u' <_D v'$ (Type 2). Given a Type 1 (or Type 2) order crossing $(u',v'),$ if both $u'$ and $v'$  begin or end in the same letter $a\in \A,$ then by cycling the letter $a$ in both $u'$ and $v'$ to the opposite end of the word, we obtain a new pair of conjugates $(u'',v'')$ which also represents a Type 1 (or Type 2) order crossing. We shall regard the pairs $(u',v')$ and $(u'',v'')$ as equivalent and let $T_1(u,v)$ be the set of equivalence classes of all Type 1 order crossings, and $T_2(u,v)$  the set of  equivalence classes of all Type 2 order crossings. Each class admits a unique representative $(u',v')$ ending (or beginning) in different letters. For the purpose of counting the number of classes of each type, we shall canonically represent each Type 1 class by a pair of conjugates $(u',v')$ ending in distinct letters, and each Type 2 class by a pair $(u',v')$ beginning in distinct letters. More precisely, write $u=u_1\cdots u_s$ and $v=v_1\cdots v_r$ and   set $R=\{(i,j)\,|\, 1\leq i\leq r, 1\leq j\leq s\}. $ Define sets
\begin{equation}\label{type1}T_1(u,v) =\{(i,j)\in R \, |\, v_i\neq u_j \,\,\mbox{and}\,\, u_{j+1}\cdots u_j<_A v_{i+1}\cdots v_i <_D u_{j+1}\cdots u_j\}\end{equation} and
\begin{equation}\label{type2}T_2(u,v) =\{(i,j)\in R  \,|\,  v_{i+1}\neq u_{j+1} \,\,\mbox{and}\,\, v_{i+1}\cdots v_i<_A u_{j+1}\cdots u_j <_D v_{i+1}\cdots v_i\}.\end{equation}
The sets $T_1(u,v)$ and $T_2(u,v)$ are disjoint and each $(i,j)\in T_1(u,v)$ constitutes a Type 1 order crossing  and each $(i,j)\in T_2(u,v)$ a Type 2 order crossing.

We define the {\it index} of $u$ and $v$ \[i(u,v) = |T_1(u,v)| + |T_2(u,v)|.\]
Thus,  if $\mathcal F=(u,v)$ is $\pi$-clustering, then $T_1(u,v)$ and $T_2(u,v)$ are each empty and $i(u,v)=0.$ 
While if $\mathcal F=(u,v)$ is not $\pi$-clustering, then either $T_1(u,v)$ or $T_2(u,v)$ (or both) is non-empty and $i(u,v)>0.$ 
We note that our definition of the sets $T_1(u,v)$ and $T_2(u,v),$ and hence the index, extend to arbitrary words $u$ and $v$ over $\A.$ However, if $u$ or $v$ is not $\pi$-clustering, then  $i(u,v)$ only detects order crossings between a conjugate of $u$ and a conjugate of $v$ ignoring potential order crossings occurring strictly within $u$ or strictly within $v.$

Alternatively, in defining the index of two $\pi$-clustering words $u$ and $v,$ we could chose to count the number of violations to the $\pi$-order condition in  $L_u\cup L_v.$
By a violation to the $\pi$-order condition in $L_u\cup L_v$ we mean an occurrence of $axb$ in $u^\omega$ and $a'xb'$ in $v^\omega$ such that either $a<_A a'$ and $b'<_D b$ (Type 1) or $a'<_A a$ and $b <_D b'$ (Type 2). In this case there is no need to introduce an equivalence relation. 

\begin{proposition} Let $u$ and $v$ be $\pi$-clustering words. Then $|T_i(u,v)|$ is equal to the number of violations in $L_u\cup L_v$ to the $\pi$-order condition of Type i for each $i=1,2.$  
\end{proposition} 

\begin{proof} Each Type 1 order crossing is canonically represented by a pair of conjugates $u'$ of $u$ and $v'$ of $v$ ending in distinct letters $a$ and $a'$ respectively, and such that $u'<_A v'<_D u'.$ 
Also we can write $u'^\omega =xbz$ and $v'^\omega=xb'z'$ with $b'<_Db.$ Hence this defines an occurrence of $axb\in L_u$ and an occurrence of $a'xb'\in L_v$ with $a<_A a'$ and $b'<_D b$ as required. Conversely, 
an occurrence of $axb$ in $u^\omega$ and $a'xb'$ in $v^\omega$ with $a<_A a'$ and $b'<_D b$ naturally defines a pair of conjugates $u'$ of $u$ and $v'$ of $v$ terminating in $a$ and $a'$ respectively such that $u'<_A v'<_D u'.$ It follow that the pair $(u',v')$ is a canonical representative of a Type 1 order crossing. The Type 2 case is analogous. \end{proof}

We say $u$ and $v$ have {\it mixed order type} if $T_1(u,v)$ and $T_2(u,v)$ are each non-empty.  

\begin{example}\label{Simon}\rm On the ternary alphabet $\{1,2,3\}$ with symmetric orders  $1<_D 2<_D 3$ and $3<_A 2<_A 1,$ the {\it Simon words} $u=121313$ and $v=1222$ have mixed order type. In fact, we have a Type 1 order crossing (canonically) represented by the pair $(131213, 1222)$ since  $131213<_A 1222<_ D 131213.$  Similarly we have a Type 2 order crossing (canonically) represented by the pair $(131312, 2122)$ since $2122<_A 131312<_D 2122.$ These are the only two classes of order crossings and hence $i(u,v)=2.$
In Figure 1 we represent the $10$ conjugates of $u$ and $v$ by horizontal line segments. The conjugates of $u$ are in blue, and those of $v$ in red. The first and last letters of each conjugate are shown in the left and right columns respectively.  In the left column the conjugates are ordered according to $<_D$ and in the right column according to $<_A.$ The first pair of crossings corresponds to the (canonical) Type 1 order crossing $(131213, 1222)$ while the second pair  corresponds to the (canonical) Type 2 order crossing $(131312, 2122).$ The next pair of crossings is also of Type 2 and equivalent to $(131312, 2122)$ while the last crossing is of Type 1 and equivalent to $(131213, 1222).$

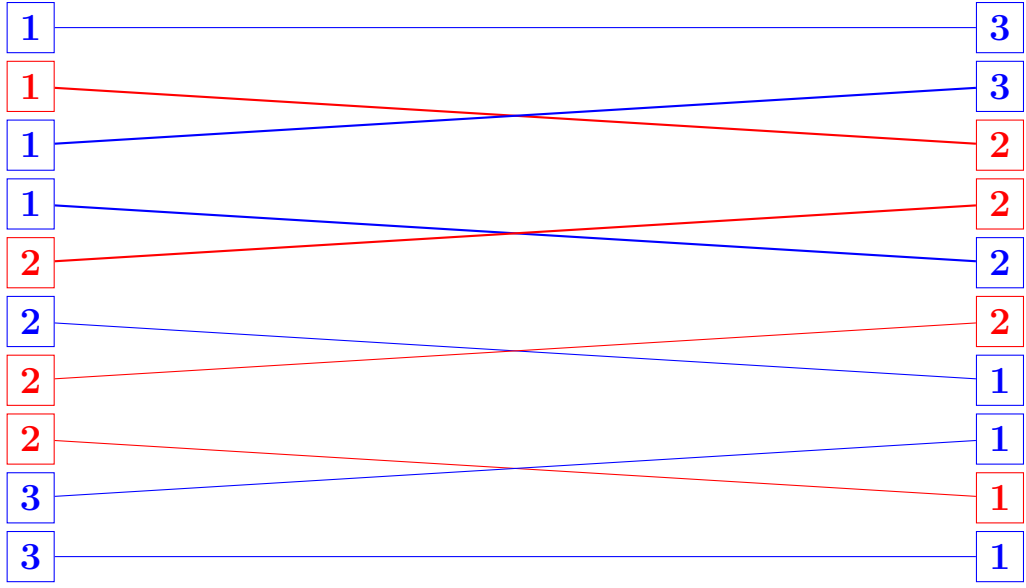
\begin{figure}\label{SimonFig}

\tikzset{
node font=\large\bfseries\boldmath\sffamily,
every node/.append style={minimum size=1mm},
}
\tikzset{
u/.style={blue,draw},
v/.style={red,draw}
}
\begin{tikzpicture}[
  minimum size=5mm,
  node distance=1mm and 12.5cm,
  >=stealth,
  bend angle=45,
  auto
]

\node[u](1) {$1$};
\node[v,below=of 1](2){$1$};
\node[u,below=of 2](3){$1$};
\node[u,below=of 3](4){$1$};
\node[v,below=of 4](5){$2$};
\node[u,below=of 5](6){$2$};
\node[v,below=of 6](7){$2$};
\node[v,below=of 7](8){$2$};
\node[u,below=of 8](9){$3$};
\node[u,below=of 9](10){$3$};

\node[u,right=125mm=of 1](1'){$3$};
\node[u,below=of 1'](2'){$3$};
\node[v,below=of 2'](3'){$2$};
\node[v,below=of 3'](4'){$2$};
\node[u,below=of 4'](5'){$2$};
\node[v,below=of 5'](6'){$2$};
\node[u,below=of 6'](7'){$1$};
\node[u,below=of 7'](8'){$1$};
\node[v,below=of 8'](9'){$1$};
\node[u,below=of 9'](10'){$1$};

\path (1) edge[blue] (1');
\path (2) edge[red, thick] (3');
\path (3) edge[blue, thick] (2');
\path (4) edge[blue, thick] (5');
\path (5) edge[red, thick] (4');
\path (6) edge[blue] (7');
\path (7) edge[red] (6');
\path (8) edge[red] (9');
\path (9) edge[blue] (8');
\path (10) edge[blue] (10');
\end{tikzpicture} 
\caption{Order crossings for $u=121313$ and $v=1222$}
\end{figure}

\end{example}

Each finite Sturmian word is a perfectly clustering word \cite{mrs}. The symmetric order condition in this case is equivalent to the balance condition. 
 
\begin{lemma}\label{sturm1}Let $u,v \in \{0,1\}^*$ be two finite Sturmian words.  Then $u$ and $v$ do not have mixed order type, i.e., if $i(u,v)>0,$ then all order crossings are of the same type.\end{lemma}

\begin{proof} Since $u$ and $v$ are Sturmian, the languages $L_u$ and $L_v$ are each balanced, that is, if $axa, bxb $ both belong to $L_u$ (respectively $L_v)$,  then $a=b.$ Suppose to the contrary that $u$ and $v$ have mixed order type. Then there exist $x$ and $y$ such that $0x0, 1y1 \in L_u$ and $1x1, 0y0 \in L_v.$ Clearly $x\neq y.$ If $y$ is a proper prefix of $x$ then since $1y1\in L_u$ it follows that $x$ must begin in $y1$ (otherwise $0y0\in L_u).$  But then in $L_v$ we have $1y1$ and $0y0,$ a contradiction. Similarly $x$ cannot be a proper prefix of $y.$ So we can write $x=zax'$ and $y=zby'$ for some $z, x', y'$ possibly empty and $a\neq b.$ Since both $0za$ and $1zb$ belong to $L_u,$  we must have $a=1$ and $b=0. $ But then in $L_v$ we have $1z1$ and $0z0,$ a contradiction.  
\end{proof}  

%While the index $i(u,v)$ depends on the choice of $u$ and $v,$ as we shall see the difference $|T_2(u,v)|-|T_1(u,v)|$ depends only on the Parikh vectors of $u$ and $v.$ 

Let    $\tau$ denote the endomorphism of  $\reals^{\A}$  given by 
\[ \mathbf u=(u_i)_{i\in \A} \mapsto \tau (\mathbf u)=(\tau_i(\mathbf u))_{i\in \A}\] where 
\[\tau_i(\mathbf u)=\sum_{\substack{j<_A i}}u_j
        -\sum_{\substack{j<_D i}}u_j.\]
        
 \begin{remark}\label{gcd}\rm  If $T$ is a $k$-interval exchange transformation with length vector  $\mathbf u$  and pair of orders $\pi,$ then  $\tau_i(\mathbf u)$ is equal to the net translational shift of the interval labeled $i$ under the interval exchange. If $T$ is a discrete $k$-interval exchange transformation, then both the length vector $\mathbf u$   and the translation vector $\tau (\mathbf u)$ are integer vectors. In the discrete case, if $T$ is minimal, then $\gcd \tau (\mathbf u)=1;$ the converse is not true in general. In this case we also have  $\gcd \mathbf u =1,$ since if  $\gcd \mathbf u \neq 1,$ then $\gcd \tau (\mathbf u)\neq 1.$% If $\mathbf u=\lambda(u)$ is the Parikh vector of some  $\pi$-clustering word $u$ over $\A,$ then the integer $\tau_i(\mathbf u)$ corresponds to the net translation of the $i$-block in the first and last columns of the Burrows-Wheeler array of $u.$  
 \end{remark}

The matrix of $\tau$ in the standard basis is the $k\times k$ skew-symmetric  matrix $\Omega=(\Omega_{i,j})_{i,j\in \A}$  given by (\ref{Omegaij}). 
The matrix $\Omega$ defines a skew-symmetric bilinear form  $\Delta : \reals ^{\A} \times \reals ^{\A} \rightarrow \reals$  given by 
\[\Delta(\mathbf u, \mathbf v)=\mathbf u^T \Omega \mathbf v=\mathbf u ^T\tau(\mathbf v) \]

    \begin{proposition}\label{invert}  $\det \Omega =0$  whenever $k$ is odd. If $\pi$ is symmetric, then  the rank of $\Omega$ is equal to $k$ for each $k$ even, and $k-1$ for each $k$ odd. Thus in the symmetric case $\Delta$ is non-degenerate  if and only if  $k$ is even. 
\end{proposition}

\begin{proof} By Jacobi's theorem $\det \Omega= 0$ for $k$ odd. In the symmetric case, the matrix   $\Omega$ is the $k\times k$ matrix with $0$'s on the main diagonal, $1$'s above the diagonal, and $-1$'s below the diagonal. It is easily seen by induction on $k$ that the rank of such matrix is $k$ for each $k$ even and $k-1$ for each $k$ odd.  \end{proof}

\noindent {\bf Digression to Rauzy classes} 
The {\it Rauzy moves} are defined  for permutations in \cite{rau}, and given by formulas $(4)$ and $(6)$ in chapter 2 of \cite{via}. In terms of pairs of orders they are described as follows: Let $M_D$ and  $M_A$ denote the maximal letters with respect to $<_D$ and $<_A$ respectively. We represent each order by listing the letters in $\A$ in increasing order so we can write  
$<_D=xM_AyM_D$ and $<_A=x'M_Dy'M_A.$ 
Define the Rauzy moves $R_t$ and $R_b$ by 
 \[R_t(<_D,<_A)=R_t( xM_AyM_D,  x'M_Dy'M_A)=(xM_AyM_D, x'M_DM_Ay')\] and
\[R_b(<_D,<_A)=R_b( xM_AyM_D,  x'M_Dy'M_A)=(xM_AM_Dy, x'M_Dy'M_A).\] 
The orbit of $\pi=(<_D,<_A)$ under de action of $R_t$ and $R_b$ is called the {\it Rauzy class} of $\pi.$ 

The rank of $\Omega$ depends only on the Rauzy class of  $\pi=(<_D,<_A).$  In fact, let $(<'_D,<'_A)$ be another pair of orders on $\A$ and let $\tau'$ denote the endomorphism of $\reals^{\A}$ given by 
\[ \mathbf u=(u_i)_{i\in \A} \mapsto \tau' (\mathbf u)=(\tau'_i(\mathbf u))_{i\in \A}\] with
\[\tau'_i(\mathbf u)=\sum_{\substack{j<'_A i}}u_j
        -\sum_{\substack{j<'_D i}}u_j,\] and $\Omega'$ denote the matrix of $\tau'$ in the standard basis. 
        Then by Lemma~10.2 in \cite{via}, if 
$(<'_D,<'_A)=R_t(<_D,<_A),$ then  $\Omega'=\Theta \Omega \Theta ^T$ where $\Theta$ is the invertible $k\times k$ matrix given by
\begin{align}\Theta_{i,j}= \begin{cases} 1\,\,\,\, \text{if  $i= j$}\\ 
1\,\,\,\, \text{if  $i=M_A$ and $j=M_D.$}\\0\,\,\,\,\, \text{otherwise} \end{cases}
\end{align}   
 Similarly, if $(<'_D,<'_A)=R_b(<_D,<_A),$ then $\Omega'=\Theta \Omega \Theta ^T$ where $\Theta$ \begin{align}\Theta_{i,j}= \begin{cases} 1\,\,\,\, \text{if  $i= j$}\\ 
1\,\,\,\, \text{if  $i=M_D$ and $j=M_A$}\\0\,\,\,\,\, \text{otherwise.} \end{cases}
\end{align}         

\vspace{.15 in}

 %For $\mathbf u, \mathbf v \in \reals^k,$ set  \[\Delta(\mathbf u, \mathbf v)=\mathbf u^T M_{\delta} \mathbf v=\mathbf u \cdot \delta(\mathbf v) \]

%\begin{lemma}\label{Delta}$\Delta (\cdot\,, \cdot) :\reals^k \times \reals^k \rightarrow \reals$  is  a skew symmetric bilinear form. In particular $\Delta (\mathbf u,\mathbf u)=0$ for all $\mathbf u\in \reals ^k.$ \end{lemma}

The above definitions and notations shall be extended to the context of words as follows: Given a word $v\in \A^*,$  let $\lambda (v)=(|v|_i)_{i\in \A}$  be the Parikh vector or length vector of $v$ in $\reals^\A$  where $|v|_i$ denotes  the number of occurrences of the letter $i\in v.$ Similarly, we let 
$ \tau (v)=(\tau_1(v), \tau_2(v), \ldots ,\tau_k(v))$ denote the {\it translation vector} of $v$ where
\[\tau_i(v)=\sum_{\substack{j<_A i}}|v|_j
        -\sum_{\substack{j<_D i}}|v|_j.\]
For words $u, v\in \A^*,$ put  \[\Delta (u,v)=\lambda(u)^T\Omega \lambda(v)=\lambda (u)^T \tau(v)=\sum_{\substack{i\in \A}}|u|_i\tau_i(v),\]

\begin{example}\rm For the symmetric pair of orders on $2$ letters,  if the Parikh vectors of $u,v $ are $(a,b)$ and $(c,d)$ respectively, then   
     \begin{equation}\label{ap}   \Delta(u,v) =\lambda (u)\cdot \tau(v)=(a,b)\cdot (d,-c)=ad-bc=\det \begin{pmatrix} a &c\\b & d \end{pmatrix}.\end{equation}
     For the symmetric pair of orders and $k=3,$ writing   $\lambda (u)=(a,b,c)$ and $\lambda (v) =(d,e,f)$  we find
     \begin{equation}   \Delta(u,v) =(a,b,c)\cdot (e+f, f-d, -d-e) =\det \begin{pmatrix} +1&a& d\\-1&b&e\\ +1&c&f \end{pmatrix}.\end{equation} 
If $k\geq 4$ and $\pi$  is symmetric,   then $\Delta(u,v)$ cannot be expressed as the determinant of a $k\times k$ matrix $A$ two of whose columns are $\lambda(u)$ and $\lambda(v)$ (as for $k=2,3$ above).
 In fact, if $x$ and $y$ are any two columns of $A$ different from $\lambda(u)$ and $\lambda(v),$ then $x$ and $y$ span a $2$-dimensional subspace of $\ker (\Omega),$ thereby contradicting Proposition~\ref{invert}.      \end{example}
     
\noindent The following is a reformulation of  Theorem~\ref{tA}:

\begin{theorem}\label{t1}Let $u$ and $v$ be words over the alphabet $\A.$ Then \[\Delta(u,v)=|T_2(u,v)|-|T_1(u,v)|.\] 
\end{theorem}

\begin{proof} We shall prove equivalently that $\Delta(v,u)=|T_1(u,v)|-|T_2(u,v)|.$ 
Write $u=u_1\cdots u_s$ and $v=v_1\cdots v_r.$ For each $(i,j)\in R$ let 
\begin{align}a_{i,j}= \begin{cases} 1\,\,\,\, \text{if  $u_j<_A v_i$}\\ 0\,\,\,\, \text{otherwise} \end{cases}
\,\,\,d_{i,j}= \begin{cases} 1\,\,\,\, \text{if  $u_j<_D v_i$}\\ 0\,\,\,\, \text{otherwise} \end{cases}
\end{align}
and $\delta_{i,j} \in \{-1,0,+1\}$ be given by \[\delta_{i,j}=a_{i,j} - d_{i+1,j+1}.\]
Then
\begin{align*}
\Delta(v,u) & =\sum _{i=1}^r \tau_{v_i}(u) \\
&= \sum_{i=1}^r (\sum_{a<_A v_i}|u|_a -\sum_{a<_D v_i}|u|_a)\\
&= \sum_{i=1}^r (\sum_{a<_A v_i}|u|_a -\sum_{a<_D v_{i+1}}|u|_a)\\
&= \sum_{(i,j)\in R} (a_{i,j} - d_{i+1,j})\\
&=  \sum_{(i,j)\in R} (a_{i,j} - d_{i+1,j+1})\\
&=  \sum_{(i,j)\in R} \delta_{i,j}
\end{align*}
where the $i$ indices are taken modulo $r,$ the $j$ indices modulo $s.$ In other words, $\Delta(v,u)$ is the sum of all 
$\delta_{i,j}\in \{-1,0,+1\}$ over all $(i,j)\in R.$ We shall partition this sum into disjoint sums of the form $\sum _{n\in \ints} \delta_{i+n,j+n}$ along lines of slope $-1.$  For this reason, for each $(i,j)\in R$ we consider the  sequence $s(i,j)=(\delta_{i+n,j+n})_{n\geq 0}.$ \\

\begin{lemma}\label{code}Fix $(i,j)\in R.$  Then
\begin{enumerate}
\item $(i,j)\in T_1(u,v)$ if and only if either $s(i,j)$ begins with a block of the form $+10^n+1$ for some $n\geq 0$ or $s(i,j)$ begins with a block of the form
 $+10^n-1$ for some $n\geq 0$ and $(i+n+1,j+n+1)\in T_2(u,v).$
\item $(i,j)\in T_2(u,v)$ if and only if either $s(i,j)$ begins with a block of the form $-10^n-1$ for some $n\geq 0$ or $s(i,j)$ begins with a block of the form
 $-10^n+1$ for some $n\geq 0$ and $(i+n+1,j+n+1)\in T_1(u,v).$
\end{enumerate}
\end{lemma}

\begin{proof} We prove only (1) as the proof of (2) is symmetric. Let us first assume that $(i,j)\in T_1(u,v).$ Then  $u_j\neq v_i$ and \begin{equation}\label{eq1}u_{j+1}\cdots u_j<_A v_{i+1}\cdots v_i <_D u_{j+1}\cdots u_j.\end{equation}
It follows that $u_j<_A v_i$ and $v_{i+1}\leq _D u_{j+1}.$ This implies that $a_{i,j}=1,$  $d_{i+1,j+1}=0 $ and hence $\delta_{i,j}=+1.$ Thus the  sequence $s(i,j)$ begins in $+1.$ If $s(i,j)$ begins in a block of the form $+10^n+1$ then there is nothing further to prove. So assume that $s(i,j)$ begins in a block of the form $+10^n-1.$ We must show that in this case $(i+n+1,j+n+1)\in T_2(u,v).$ As $s(i,j)$ begins in $+10^n-1$ we have that $\delta_{i,j}=+1,$ $\delta_{i+m,j+m}=0$ for $m=1,\ldots ,n$ and $\delta_{i+n+1,j+n+1}=-1.$  Now
$\delta_{i,j}=+1 \Leftrightarrow a_{i,j}=+1$ and $d_{i+1,j+1}=0$ and hence
\begin{align} \delta_{i,j}=+1 \Leftrightarrow \begin{cases}u_j<_A v_i\\v_{i+1}\leq_D u_{j+1}.\end{cases}\end{align}

Similarly 
\begin{align}\delta_{i+n+1,j+n+1}=-1\Leftrightarrow \begin{cases} v_{i+n+1}\leq_A u_{j+n+1}\\ u_{j+n+2}<_Dv_{i+n+2}.\end{cases}\end{align}
On the other hand, $\delta_{i+m,j+m}=0$ gives rise to two cases depending on whether $a_{i+m,j+m}$ and $d_{i+m,j+m}$ are both equal to $0$ or both equal to $1.$ We find
\begin{align} \delta_{i+m,j+m}=0 \Leftrightarrow \mbox{case (a)}\,\begin{cases}u_{j+m}<_A v_{i+m}\\ u_{j+m+1}<_D v_{i+m+1}\end{cases}\,\,\,\,\mbox{case (b)}\begin{cases}v_{i+m}\leq _A u_{j+m}\\ v_{i+m+1}\leq_D u_{j+m+1}\end{cases}
\end{align}
We begin with $\delta_{i+n+1,j+n+1}=-1.$ Since $u_{j+n+2}<_D v_{i+n+2}$ it follows that $u_{j+n+2}\cdots u_{j+n+1}$ and $v_{i+n+2}\cdots v_{i+n+1}$ begin in different letters and $u_{j+n+2}\cdots u_{j+n+1}<_D v_{i+n+2}\cdots v_{i+n+1}.$ 
To prove that $(i+n+1,j+n+1)\in T_2(u,v)$ it remains to show that $v_{i+n+2}\cdots v_{i+n+1} <_A u_{j+n+2}\cdots u_{j+n+1}.$
If $v_{i+n+1}<_A u_{j+n+1},$ then $v_{i+n+2}\cdots v_{i+n+1}<_A u_{j+n+2}\cdots u_{j+n+1}$ as required. Thus we may assume that
$v_{i+n+1}=u_{j+n+1}.$ If  $v_{i+m}=u_{j+m}$ for each $m=1,\ldots n+1,$  then since $u_{j+n+2}<_Dv_{i+n+2}$ it follows that $(u_{j+1}\cdots u_j)^\omega <_D (v_{i+1}\cdots v_i)^\omega$ and hence $u_{j+1}\cdots u_j <_D v_{i+1}\cdots v_i,$ a contradiction to (\ref{eq1}). 
Thus there exists a largest $1\leq p \leq n$ such that $v_{i+p}\neq  u_{j+p}$ and $v_{i+m}=u_{j+m}$ for each $m=p+1,\ldots ,n+1.$
Then when $\delta_{i+p,j+p}=0,$ we must be in case (b) since in case (a) we have $u_{j+p+1}<_D v_{i+p+1}$ contradicting that $u_{j+p+1}= v_{j+p+1}.$
Thus $ v_{i+p}<_A u_{j+p}.$   
It follows that $(v_{i+n+2}\cdots v_{i+n+1})^{-\omega} <_A (u_{j+n+2}\cdots u_{j+n+1})^{-\omega}$ and hence $v_{i+n+2}\cdots v_{i+n+1}<_A  u_{j+n+2}\cdots u_{j+n+1}$ as required.

We now prove the converse of (1). First assume that $s(i,j)$ begins in $(+1)0^n(+1)$ for some $n\geq 0.$ Then $\delta_{i,j}=1,$ $\delta_{i+m,j+m}=0$ for $m=1,\ldots, n$ and $\delta_{i+n+1,j+n+1}=1.$ As $\delta_{i,j}=1,$  it follows that $u_j<_A v_i$ which implies that $u_{j+1}\cdots u_j$ and $v_{i+1}\cdots v_i$ terminate in distinct letters and  $u_{j+1}\cdots u_j<_A v_{i+1}\cdots v_i.$ To show that $(i,j)\in T_1(u,v)$ it remains to show that 
$v_{i+1}\cdots v_i<_D u_{j+1}\cdots u_j.$   Since $d_{i+1,j+1}=0$ we have that $v_{i+1}\leq _D u_{j+1}.$ If $v_{i+1}<_Du_{j+1}$ then we are done. Otherwise if $v_{i+1}=u_{j+1},$
 then $a_{i+1,j+1}=0$ and hence $d_{i+2,j+2}=0$ (since $\delta_{i+1,j+1}=0).$ So $v_{i+2}\leq_D u_{j+2}.$ If $v_{i+2}<_D u_{j+2}$ then we are done, otherwise if $v_{i+2}=u_{j+2},$ then $a_{i+2,j+2}=0$ and hence $d_{i+3,j+3}=0.$  Continuing in this way, either there exists $m\leq n$ such that $v_{i+k}=u_{j+k}$ for $k=1,\ldots ,m-1$ and $v_{i+m}<_D u_{j+m}$ in which case we have $v[i+1]<_D u[j+1]$ or we have $v_{i+k}=u_{j+k}$ for each $k=1,\ldots ,n.$ In particular, $a_{i+n,j+n}=0$ and hence $d_{i+n+1,j+n+1}=0$ which implies that $v_{i+n+1}\leq_D u_{j+n+1}.$ But since $\delta_{i+n+1,j+n+1}=1,$ it follows that $a_{i+n+1,j+n+1}=1$ which implies in particular that $v_{i+n+1}\neq u_{j+n+1}.$ Hence  $v_{i+n+1}<_D u_{j+n+1}$ as required. 
 
 It remains to show that if $s(i,j)$ begins in $+10^n-1$ and $(i+n+1,j+n+1)\in T_2(u,v),$ then $(i,j)\in T_1(u,v).$ The proof here is essentially symmetric to our earlier proof that if $s(i,j)$ begins in $+10^n-1$  and $(i,j)\in T_1(u,v)$ then $(i+n+1,j+n+1)\in T_2(u,v).$ 
First since $(i+n+1,j+n+1)\in T_2(u,v)$ it follows that $v_{i+n+2}\cdots v_{i+n+1}$ and $u_{j+n+2}\cdots u_{j+n+1}$ begin in different letters and
\begin{equation}\label{con}v_{i+n+2}\cdots v_{i+n+1} <_A  u_{j+n+2}\cdots u_{j+n+1}   <_D v_{i+n+2}\cdots v_{i+n+1}.\end{equation}
This time starting with $\delta_{i,j}=+1,$ we have $u_j<_A v_i$ and $v_{i+1}\leq _D u_{j+1}.$ Thus $u_{j+1}\cdots u_j$ and $v_{i+1}\cdots v_i$ end in different letters and $u_{j+1}\cdots u_j <_A v_{i+1}\cdots v_i.$ To show that $(i,j)\in T_1(u,v)$ it remains to show that $v_{i+1}\cdots v_i <_D  u_{j+1}\cdots u_j.$ If $v_{i+1}<_D u_{j+1}$ then we are done. Thus we can assume hereon that $v_{i+1}=u_{j+1}.$ We consider two cases: Case 1,  $v_{i+m}=u_{j+m}$ for each $m=1,\ldots ,n.$ In this case, when $\delta_{i+m,j+m}=0$ we are in case (b).  In particular $v_{i+n+1}\leq _D u_{j+n+1}.$ We claim $v_{i+n+1} <_D u_{j+n+1}$ and hence that $v_{i+1}\cdots v_i <_D  u_{j+1}\cdots u_j.$ If $v_{i+1}<_D u_{j+1}.$  In fact,  if  $v_{i+n+1}= u_{j+n+1},$ then since $u_j<_A v_i$ we have that $u_{j+n+2}\cdots u_{j+n+1} <_Av_{i+n+2}\cdots v_{i+n+1}$ contradicting (\ref{con}).
Case 2, there exists a least $1<p<n$ such that $v_{i+p}\neq u_{j+p}$ and $v_{i+m}=u_{j+m}$ for $m=1,\ldots ,p-1.$ Thus when $\delta_{i+p-1,j+p-1}=0,$ we must be in case (b) and hence $v_{i+p}<_D u_{j+p}$ and hence $v_{i+1}\cdots v_i<_D u_{j+1}\cdots u_j$ as required. This completes the proof of item (1)  of the lemma. \end{proof}

Returning to the proof of Theorem~\ref{t1}, we partition the sum  $\sum_{(i,j)\in R}\delta_{i,j}$ along disjoint lines of slope $-1$ and hence depending on the greatest common divisor of $r$ and $s$ it may be one or more lines. Summing along the line $L_{i,j}$ of slope $-1$ containing $\delta_{i,j}$ amounts to summing all the terms of the sequence $s(i,j).$ By Lemma~\ref{code}, if $\delta_{i,j}=+1,$ then if the next non-zero entry along the line $L_{i,j}$ is $+1,$ then $(i,j)\in T_1(u,v),$ while if the next non-zero entry is equal to $-1$ then $(i,j)$ is in $T_1(u,v)$ if and only if the coordinate pair corresponding to $-1$ belongs to $T_2(u,v).$ In other words, $|T_1(u,v)|$ may be smaller than the sum of all the $\delta_{i,j}=+1$ since some occurrences of $\delta_{i,j}=+1$ may not correspond to a pair $(,j)$ belonging to $T_1(u,v).$ And similarly, $|T_2(u,v)|$ may be smaller than the absolute value of the sum of all the $\delta_{i,j}=-1.$ Nevertheless, summing all the $\delta_{i,j}$ is equal to $|T_1(u,v)|-|T_2(u,v)|.$ This concludes our proof of Theorem~\ref{t1}   \end{proof}

We now look at some applications of Theorem~\ref{t1}.

\begin{corollary}\label{index0} For words $u,v \in \A^*,$ if $i(u,v)=0,$ then $\Delta(u,v)=0.$
\end{corollary}

\begin{proof} If $i(u,v)=0,$ then $T_1(u,v)$ and $T_2(u,v)$ are both empty and hence $\Delta(u,v)=0$ by Theorem~\ref{t1}. 
\end{proof}

The converse to Corollary~\ref{index0} is false in general. As noted in Example~\ref{Simon}, the  words $u=121313$ and $v=1222$ are each perfectly clustering and  $i(u,v)>2.$ On the other hand $\lambda (u)=(1,3,0)$ and $\tau(v)=(3,-1,-4)$ and therefore $\Delta(u,v)=(1,3,0)\cdot (3,-1,-4)=0.$

\begin{corollary}Let $\mathcal F=(v_1,v_2,\ldots,v_n)$ be a primitive family over $\A.$ If $\mathcal F$ is $\pi$-clustering, then $n\leq \lfloor \frac {k+d}{2} \rfloor$ where $d=\dim \ker (\Omega).$  
If $\mathcal F$ is perfectly clustering, then $n\leq \lfloor \frac {k+1}{2} \rfloor,$ moreover, this  bound is optimal. 
\end{corollary}

\begin{proof}Assume $\mathcal F$ is $\pi$-clustering. Then by Corollary~\ref{span} the set of Parikh vectors
$P_{\mathcal F}= \{\lambda (v_i)\,|\,1\leq i\leq n\}$ spans an $n$-dimensional subspace $S$ of $\reals^\A.$ By Theorem~\ref{t1}, the translation vectors $ \{\tau (v_i)\,|\,1\leq i\leq n\}$ lie in the subspace orthogonal to $S$ and hence $S\cap \tau(S)=\{\mathbf 0\}.$ Moreover, $\dim \tau(S)  \geq n-d$ where $d=\dim \ker (\Omega).$ Thus $n+ (n-d) \leq \dim S + \dim \tau(S) \leq k$ or equivalently $n\leq \lfloor \frac {k+1}{2} \rfloor.$
If $\pi$ is symmetric, then by Proposition~\ref{invert} the rank of $\Omega$ is equal to $k$ for $k$ even (and hence $d=0)$ and $k-1$ for $k$ odd (and hence $d=1).$ Thus if $\mathcal F$ is perfectly clustering, it follows that $n\leq \lfloor \frac {k+1}{2} \rfloor.$ To see that this bound is optimal it suffices to consider the perfectly clustering family $\mathcal F=(1k, 2(k-1),3(k-2),\ldots ).$\end{proof}

Theorem~\ref{t1} provides a simple formula for computing the index of two $\pi$-clustering words $u$ and $v$ not having mixed order type: 
 
\begin{corollary}\label{isturm}Let $u$ and $v$ be  $\pi$-clustering words on the alphabet $\A.$ If $u$ and $v$ do not have mixed order type, then \[i(u,v)=|\Delta(u,v)|.\]
In particular, for all finite Sturmian words $u$ and $v$ with Parikh vectors $(a,b)$ and $(c,d)$ respectively,  the index $i(u,v)=|ad-bc|.$ 
\end{corollary}

\begin{proof} The first statement follows immediately from Theorem~\ref{t1} since either $T_1(u,v)$ or $T_2(u,v)$ (or both) is empty.  Finally if $u$ and $v$ are Sturmian, then  $u$ and $v$ do not have mixed order type by Lemma~\ref{sturm1} and hence $i(u,v)=|\Delta(u,v)|=|ad-bc|$ by  $(\ref{ap}).$
\end{proof} 

\begin{corollary}\label{traj} Let $T$ be a discrete interval exchange transformation with pair of orders $\pi.$ Then for all trajectories $v_i^\infty$ and $v_j^\infty$  we have $\Delta(v_i,v_j)=0$ 
\end{corollary} 

\begin{proof} For all trajectories $v_i^\infty$ and $v_j^\infty$ the family $\mathcal F=(v_i,v_j)$ is $\pi$-clustering and hence $i(v_i,v_j)=0.$ The result now follows from Corollary~\ref{index0}.
\end{proof}

\begin{corollary}\label{3-iet} Let $T$ be a discrete $3$-interval exchange transformation with length vector $\mathbf u$ and pair of orders $\pi.$ Then the number of orbits of $T$ is equal to $\gcd \tau (\mathbf u).$ In particular, $T$ is minimal if and only if $\gcd \tau(\mathbf u)=1.$
\end{corollary}

\begin{proof} Let $v_1^\infty$ and $v_2^\infty$ be  two trajectories of $T$ with $v_i$ primitive. Then we claim $\tau(v_1)=\tau(v_2).$
In fact, $v_1$ and $v_2$ are each $\pi$-clustering words and each corresponds to the unique trajectory of a minimal discrete interval exchange transformation with length vector $\lambda (v_1)$ and $\lambda (v_2)$ respectively. It follows from Remark~\ref{gcd} that $\gcd  \lambda (v_1)=\gcd \lambda (v_2)=\gcd \tau (v_1)=\gcd \tau(v_2)=1.$ 
Now if $\lambda (v_1)=\lambda (v_2)$ then  $\tau(v_1)=\Omega \lambda(v_1)=\Omega \lambda (v_2)=\tau(v_2).$ Thus we can assume that $\lambda (v_1)\neq \lambda (v_2).$ In this case, $\lambda (v_1)$ and $\lambda (v_2)$ are not collinear, otherwise dividing each by their respective $\gcd$ would give the same vector.  Let $S_1$ (respectively $S_2)$ be the $2$-dimensional subspace of $\reals^3$ perpendicular to $\lambda (v_1)$ (respectively $\lambda( v_2)).$ Then since $\lambda (v_1)$ and $ \lambda (v_2)$ are not collinear, $S_1\neq S_2$ and hence $S_1\cap S_2$ is a line. Since $\Delta (\cdot \,, \cdot)$ is skew symmetric $\Delta(v_1,v_1)=\Delta(v_2,v_2)=0$ and hence $\tau(v_1)\in S_1$ and $\tau (v_2)\in S_2.$ But also $\Delta(v_1,v_2)=\Delta(v_2,v_1)=0$ by Corollary~\ref{traj}, and hence $\tau(v_1)\in S_2$ and $\tau(v_2)\in S_1.$ It follows that $\tau (v_1), \tau(v_2) \in S_1 \cap S_2$ and are hence collinear and therefore equal. Thus if $X$ has $n$-trajectories $v_1^\infty,\ldots ,v_n^\infty$ (not necessarily distinct) with $v_i$ primitive, then by linearity of $\tau$ we have $\tau (\mathbf u)=\tau(v_1)+\cdots +\tau(v_n)=n\tau(v_1)$ and $\gcd \tau (v_1)=1.$ Thus $n=\gcd \tau (\mathbf u)$ as required.   \end{proof}

\begin{remark}\label{PakRed}\rm For $T$ a symmetric $3$-interval exchange transformation, if $\mathbf u =(a,b,c)$ then \newline $\tau (\mathbf u)=(b+c,c-a,-a-b)$ and hence $T$ is minimal $\Leftrightarrow\gcd \tau(\mathbf u)=1\Leftrightarrow \gcd (b+c,c-a,-a-b)=1 \Leftrightarrow \gcd(a+b,b+c)=1,$ a result first due to I. Pak and A. Redlich in \cite{PR}.  
\end{remark}

%\noindent We obtain the following strengthening of Corollary~\ref{ktraj} : 

\begin{corollary}\label{k-1traj} Let $T$ be a discrete $k$-interval exchange transformation with  pair of orders $\pi.$  Then $T$ has at most $\frac{k+d}{2}$ distinct trajectories where $d=\dim\ker (\Omega).$ In particular,  a symmetric $k$- interval exchange transformation has at most $\lfloor \frac {k+1}{2} \rfloor$ distinct trajectories. Finally if $T$ has $k-1$ distinct trajectories and $k\geq 3,$ then $\gcd \tau (\mathbf u) \neq 1.$   \end{corollary}

\begin{proof}  Assume $T$ has $n$ distinct trajectories  $v_1^\infty,v_2^\infty, \ldots ,v_n^\infty$ with $v_i$ primitive.  Let $S_1$ be the subspace spanned by $\lambda ( v_1),\ldots ,\lambda (v_n)$ and $S_2$  the subspace spanned by $\tau(v_1),\ldots, \tau(v_n).$ By Corollary~\ref{traj}, we have $\Delta(v_i,v_j)=0$ for each $i,j$  and hence each vector in $S_1$ is orthogonal to each vector in $S_2.$ Moreover, by Proposition~\ref{li} we have $\dim S_1=n$ and $\dim S_2\geq n-d$.   Thus $2n-d\leq k.$
Also  $\frac{k+d}{2}\leq k-1$ since $d\leq k-1$ as we are assuming $<_D\neq <_A.$ 

If $T$ is a discrete symmetric interval exchange transformation, then by Proposition~\ref{invert}, we have  $d=0$ for each $k$ even, and $d=1$ for each $k$ odd. Whence $T$ has at most $\lfloor \frac {k+1}{2} \rfloor$ distinct trajectories.

 Finally if $T$ has $k-1$ distinct trajectories $v_1^\infty,\ldots,v_{k-1}^\infty$  with $v_i$ primitive, then their Parikh vectors span a $k-1$-dimensional subspace $S$ of $\reals ^{\A}.$ Since $\Delta(v_i,v_j)=0$ for all $i,j,$ it follows that each  translation vector $\tau(v_i)$ is orthogonal to $S$ and hence the translation vectors $\tau(v_i)$ are all collinear and hence all equal (since $\gcd \tau (v_i)=1$ for each $i).$ But then $\gcd \tau (\mathbf u) =k-1 >1$ for $k\geq 3.$ \end{proof}
 
 \begin{remark}\rm If $T$ is a real continuous standard $k$-interval exchange transformation, then the rank of  $\Omega$ is identified with $2g$, where $g$ is the genus of the translation surface $M$  built from $T.$ %; this implies that the rank depends only on the {\it Rauzy class} of the pair of orders, see \ref{raucl}, and hence all our bounds apply to entire Rauzy classes. 
The number of singularities of $M$ is $j=d +1$ where $d=\dim \ker (\Omega).$     A  result attributed to Smillie, stated and proved in Theorem 2 of \cite{nav}, states that the number of distinct minimal periodic components, called {\it cylinders}, of the translation flow of $M$ is at most $g+j-1.$  Combining, we obtain that the number of distinct cylinders is at most $\frac{k+d}{2}.$  When the intervals defining $T$ have integer lengths, the surface $M$ is {\it square-tiled}. This is used in Corollary 1.15 of \cite{z1} and Theorem 1.2 of \cite{z2} to derive estimates on the number of discrete interval exchange transformations whose trajectories have certain properties. Moreover the number of distinct cylinders is also the number of distinct trajectories of the corresponding discrete interval exchange, and thus we recover in this way the upper bound  $\lfloor \frac {k+d}{2} \rfloor$ on the number of distinct trajectories. To the best of our knowledge, this formulation for this upper bound on has not been previously stated. Our proof of Theorem \ref{tA} is purely combinatorial and does not rely on the  geometric methods typically used in the study of continuous real interval exchange transformations.  
\end{remark}

\section*{Acknowledgements} 
The authors wish to thank Erwan Lanneau for   pointing to us Smillie's result and finding reference \cite{nav}.

\end{document}